\documentclass[11pt,a4paper]{amsart}
\usepackage{amsfonts}
\usepackage[active]{srcltx}
\usepackage[notref,notcite]{showkeys}
\usepackage{enumerate}

\usepackage{amsmath,amssymb,xspace,amsthm}
\newtheorem{theorem}{Theorem}
\newtheorem{corollary}[theorem]{Corollary}
\numberwithin{equation}{section}

\def\Re{\operatorname{Re}}
\def\Im{\operatorname{Im}}

\newcommand{\C}{\ensuremath{\mathbb C}\xspace}

\newcommand{\R}{\ensuremath{\mathbb{R}}\xspace}

\renewcommand{\phi}{\varphi}

\renewcommand{\leq}{\leqslant}
\renewcommand{\geq}{\geqslant}

\begin{document}
\title[simple quaternionic polynomials]{A new method of finding all Roots of simple quaternionic polynomials}
\author{Lianggui Feng and  Kaiming Zhao}
\thanks{K.M. Zhao is partially supported by NSERC and NSF of China
(Grant 10871192), and L.G. Feng is partially supported by HNSF of
China (Grant 11JJ7002).}
 \maketitle
\begin{abstract}
In this paper, we provide a new method to find all zeros of
polynomials with quaternionic coefficients located on only one side
of the powers of the variable (these polynomials are called simple
polynomials). This method is much more  efficient and much simpler
than the known one in \cite{JO1}. We  recover several known results,
and deduce several  interesting consequences concerning solving
equations with all real coefficients or complex coefficients which
do not seem to be deduced easily from the results in \cite{JO1}.
 We also give a necessary and sufficient condition for a simple quaternionic
polynomials to have finitely many solutions (only isolated
solutions).\end{abstract}


 \vskip 10pt

\section{Introduction}

The quaternion algebra plays an important role in many subjects,
such as, quaternionic quantum mechanics \cite{A}, and signal
precessing \cite{BM,JO5}. Because of the noncommutativity of
quaternionic multiplication, solving a quaternionic equation of
degree $n$ becomes a challenging problem [3-19].

\vskip 5pt
 Niven in \cite{N,N2} made first steps in generalizing the
fundamental theorem  of algebra onto quaternionic situation which
led to the article by Eilenberg and Niven \cite{EN} where the
existence of roots for a quaternion equation of degree $n$ was
proved using strongly topological methods. After that, Topuridze in
\cite{T}, also with help of topological method,  showed that the
zero set of  polynomials with quaternionic coefficients located on
only one side  of the power of the single variable (these
polynomials are called {\it simple polynomials}) consists of a
finite number of points and Euclidean spheres of corresponding
dimension.

\vskip 5pt Concerning about the computation of roots of a
quaternionic polynomial, the first numerically working algorithm to
find a root was presented by Ser$\hat{o}$dio, Pereira, and
Vit$\acute{o}$ria \cite{SPV}, and further contributions were made by
Ser$\hat{o}$dio and Siu \cite{SS}, Pumpl$\ddot{u}$n and Walcher
\cite{PW}, De Leo, Ducati, and Leonardi \cite{DDL}, Gentili and
Struppa \cite{GS}, Gentili, Struppa, and Vlacci \cite{GSV},  Gentili
and Stoppato \cite{GSt}. A large bibliography on quaternions in
general was given by Gsponer and Hurni in 2006 \cite{GH}. Recently,
Janovsk$\acute{a}$ and Opfer presented a  method in \cite{JO1} for
producing all zeros of a simple quaternionic polynomial, by using
the  real companion polynomial introduced for the first time by
Niven \cite{N}, the number one introduced in \cite{PS}, and the
presentation of the powers of a quaternion as a real, linear
combination of the quaternion. Let us recall some notions first.

\vskip 5pt Throughout this paper, let $\mathbb{N}$  be the set of
positive integers, $\mathbb{R}$  the real number field,
$\mathbb{C}=\mathbb{R}\oplus \mathbb{R}\mathbf{i}$  the complex
number field, and $\mathbb{H}$ the skew-field of real quaternions,
that is, any element of $\mathbb{H}$ is of the form $ q=a_0+a_1
\mathbf{i}+a_2 \mathbf{j}+a_3\mathbf{k} =(a_0+a_1 \mathbf{i})+(a_2
+a_3 \mathbf{i})\mathbf{j}$, where
$\mathbf{i},\mathbf{j},\mathbf{k}$ are usual quaternionic imaginary
units, and $a_0,a_1,a_2,a_3\in \mathbb{R}$, and the $\R$-bilinear
product is determined by
$\mathbf{i}^2=\mathbf{j}^2=\mathbf{k}^2=-1,$
$\mathbf{i}\mathbf{j}=-\mathbf{j}\mathbf{i}=\mathbf{k},\quad
\mathbf{j}\mathbf{k}=-\mathbf{k}\mathbf{j}=\mathbf{i},\quad
\mathbf{k}\mathbf{i}=-\mathbf{i}\mathbf{k}=\mathbf{j}.$
 For
 the above quaternion $q$, we denote by $\Re q$ the real part
of $q$, by $|q|$ the module of $q$ (i.e.
$|q|=\sqrt{a_0^2+a_1^2+a_2^2+a_3^2}$), and by $[q]$
  the {\it conjugate class of} $q$(i.e.
$[q]=\{aqa^{-1}|a\in\mathbb{H},a\neq 0\}$).  We  confirm that  a
quaternionic polynomial with the coefficients on the same side  of
the power of the single variable are called
 {\it a simple quaternionic polynomial} or a {\it simple polynomial}.

 \vskip 5pt {\bf Definition 1.1.} Let $z_0\in \mathbb{H}$ be a zero of a simple polynomial $p(z)$
as given on the left-hand side in (2.1) (i.e., $p(z_0)=0$). If $z_0$
is not real and has the property that $p(z) = 0$ for all $z \in
[z_0]$, then we will say that $z_0$ is  a {\it spherical zero of
$p(z)$}. If $z_0$ is real or is not a spherical zero, it is called
an {\it isolated zero of $p(z)$}.

 \vskip 5pt In the present paper, we provide a new  method for finding all
zeros of simple polynomials $p(z)$ of arbitrary degree $n$ (Theorems
1 and 4). Our proof is based on two well-known techniques: the
presentation of a quaternion as a $2\times2$ complex matrix, and the
Jordan canonical form of a complex matrix. We first write $p(z)$ so
that its constant term is $1$ or $0$. Then introduce {\it derived
polynomials} $f_1(t)$ and $f_2(t)$ of $p(z)$ which have complex
coefficients, where $t$ is a real variable, such that
$p(t)=f_1(t)+f_2(t)\mathbf{j}$; and  define the {\it discriminant
polynomial} $\tilde
p(t)=f_1(t)\overline{f}_1(t)+f_2(t)\overline{f}_2(t)$ of $p(z)$
which is a polynomial with real coefficients, where $t$ is
considered as a complex variable, the polynomials
$\overline{f}_1(t)$ and $\overline{f}_2(t)$ are obtained by only
taking the conjugate coefficients of $f_1(t)$ and $f_2(t)$
respectively. Then all zeros of $p(z)$ can be obtained from complex
zeros of the discriminant polynomial $\tilde p(t)$. More precisely,
let $z_0\in\C$ such that $\tilde p(z_0)=0$. If $z_0$ is real then it
is an isolated zero of $p(z)$. If $z_0$ is not real and
$f_1(z_0)=f_2(z_0)=\bar f_1(z_0)=\bar f_2(z_0)=0$ then it is a
spherical zero of $p(z)$. If $z_0$ is not real and at least one of
$f_1(z_0), f_2(z_0), \bar f_1(z_0), \bar f_2(z_0)$ is not zero,  let
$\left( \begin{array}{cc} a\\ b
 \end{array}  \right)$ be a unit complex solution of the linear system
$$\left( \begin{array}{cc} f_2(\overline{z_0}) &
f_1(\overline{z_0}) \\ \bar f_1(\overline{z_0}) & -\bar f_2(\overline{z_0}) \\
 \end{array}  \right)X=0.$$ Then corresponding to this $z_0$ we have an
 isolated solution for $p(t)$:
$${|a|^2z_0+|b|^2\overline{z_0}-2b}\overline{a}(\Im
z_0)\mathbf{k}.$$ From the above three cases we obtain all zeros of
$p(z)$.

\ The paper is organized as follows. In Section 2, we prove that our
methods for solving a simple polynomial equation in Theorems 1 and 4
are valid. Then we  recover several known results, and deduce
several very interesting consequences concerning solving equations
with all real coefficients or complex coefficients which do not seem
to be deduced easily from the results in \cite{JO1} (see Corollary
\ref{corollary2} and Corollary \ref{corollary3}). We also give a
necessary and sufficient condition for a simple quaternionic
polynomials to have  finitely many solutions (only isolated
solutions). In Section 3 we give an algorithm to find all zeros of a
simple quaternionic equation, based upon our Theorems 1 and 4. In
Section 4 we give three examples. In particular, we use our method
to redo Example 3.8 in \cite{JO1}. In Section 5, we make some
necessary numerical considerations and compare our algorithm with
that in \cite{JO1}.

\section{Finding all zeros of a simple
polynomial}

We consider the simple quaternionic polynomial equation:
\begin{equation}\label{eq:simple}
q_nx^n+\cdots+q_1x+q_0=0\,\,(q_n\neq 0),
\end{equation}
where $x\in\mathbb{H}$ is the variable, $n\in \mathbb{N}$ and $q_i
\, (i=0,\ldots,n)\in\mathbb{H}$ are given. If $q_0\neq 0$, then
Eq.(\ref{eq:simple}) can be written as
  $$q_0^{-1}q_nx^n+\cdots+q_0^{-1}q_1x+1=0.$$
  Hence, in order to solve Eq.(\ref{eq:simple}), it suffices to solve the following equation
\begin{equation}\label{eq:simple 1}
p_nx^n+\cdots+p_1x+d_0=0,
\end{equation}
where $p_i \, (i=1,\ldots,n)\in\mathbb{H}, p_n\neq 0, d_0=0 \
{\text{or}} \ 1$. We simply denote the left-hand side of (2.2) as
$p(x)$.

 Let $\sigma:\mathbb{H}\to\mathbb{C}^{2\times 2}$,
 $q=z_1+z_2\mathbf{j}\mapsto \left(
 \begin{array}{cc}
  z_1 & z_2 \\
  -\overline{z_2} &\overline{ z_1} \\
 \end{array}
\right),$ where $z_1,z_2\in\mathbb{C}$. Then $\sigma$ is an
$\R$-algebra monomorphism from $ \mathbb{H}$ to $\mathbb{C}^{2\times
2}$. Sometimes, this monomorphism is  also named as the derived
mapping of $\mathbb{H}$, and  $\sigma(q)$ is denoted by $q^\sigma$.
Obviously,   $a^\sigma=\left(  \begin{array}{cc}
 a & 0 \\ 0 & a \\   \end{array} \right)$ for any $a\in \mathbb{R}$.

 Let  $p_i^\sigma=\left( \begin{array}{cc} t_1^{(i)} & t_2^{(i)} \\
  -\overline{t_2^{(i)}} & \overline{t_1^{(i)}} \\   \end{array}
  \right)$ for
  $i=1,\ldots,n,$
  where  $t_1^{(i)},  t_2^{(i)}\in  \mathbb{C}$.  Then (\ref{eq:simple 1})  becomes
  the  following  matrix  equation  in the matrix variable $Y=\left(
 \begin{array}{cc}
 z_1 & z_2 \\
 -\bar{z_2} &\bar{ z_1} \\
 \end{array}
  \right)\,\,(z_1,z_2\in\mathbb{C}$):
  \begin{equation}\label{eq:matrix}
  \begin{array}{c}
   \left(
  \begin{array}{cc}
 t_1^{(n)} & t_2^{(n)} \\
 -\overline{t_2^{(n)}} & \overline{t_1^{(n)}} \\
  \end{array}
  \right)Y^n+\cdots+\left(
  \begin{array}{cc}
 t_1^{(1)} & t_2^{(1)} \\
 -\overline{t_2^{(1)}} & \overline{t_1^{(1)}} \\
  \end{array}
  \right)Y+\left(
  \begin{array}{cc}
 d_0 &0  \\
 0 & d_0 \\
  \end{array}
  \right)=0.\end{array}
  \end{equation}

  \vskip 5pt
  Now we  introduce  a matrix  polynomial  $P(t)$  in real
  variable  $t$ (considered as a real number) as  follows:
  \begin{equation}\label{P(t)} P(t)\equiv
  p_n^\sigma
  t^n+\cdots+p_1^\sigma
  t+d_0I,\end{equation}
  where $I$
  is the $2\times2$
  identity
  matrix.

Write
  $P(t)=Q(t)(tI-Y)+P_l(Y)$,
  where
  $$P_l(Y)= p_n^\sigma
  Y^n+\cdots+p_1^\sigma
  Y+d_0I,$$
  $$
  Q(t)=p_n^\sigma
  t^{n-1}+(p_{n-1}^{\sigma}+p_n^\sigma
Y)t^{n-2}+\cdots+(p_1^{\sigma}+p_2^{\sigma}Y+\cdots
+p_n^{\sigma}Y^{n-1}).$$
  If  $P_l(Y)=0$,  then  $P(t)=Q(t)(tI-Y)$,  and hence
 $$
 \det P(t)=\det Q(t)\det (tI-Y)=\chi_Y(t)\det
 Q(t),$$
 where $\chi_Y(t)$ is the characteristic polynomial of $Y$.

 Set $\tilde p(t)\equiv \det P(t)$. By Cayley-Hamilton
 Theorem, we see that $\tilde p(Y)=0$ for every $Y$ satisfying $P_l(Y)=0$.

  Notice
  that, $\tilde p(t)\equiv
 \det P(t)$ is
 a polynomial in real variable $t$ of degree $2n$ with real
 coefficients, since $\tilde p(t)=\det (p_nt^n+\cdots+p_1t+d_0)^\sigma$, and
 $\det (p_nt^n+\cdots+p_1t+d_0)^\sigma\geq
 0$ for any real value of $t$. Then
\begin{equation}\label{p(t)}\tilde p(t)=b{(t-\xi_1)^{2r_1}\cdots(t-\xi_s)^{2r_s}}
\cdot\hskip 2cm \end{equation}
$$\hskip 2cm \cdot{(t-\eta_1)^{s_1}(t-\overline{\eta_1})^{s_1}}
\cdots{(t-\eta_k)^{s_k}(t-\overline{\eta_k})^{s_k}},$$ where $b$ is
a real number, $\xi_1,\ldots,\xi_s$ are distinct real numbers,
$\eta_1,\bar \eta_1,\ldots,$ $ \eta_k,\bar \eta_k$ are distinct
nonreal complex numbers, and $r_1,r_2,$ $\cdots, r_s, s_1,\cdots, $
$s_k\in \mathbb{N}$. It is clear that $s+k\le n$.

Now suppose $Y$ is a solution of Equation (\ref{eq:matrix}). Then,
$\tilde p(Y)=0$.
  Since $Y$  is  of form $\left(
 \begin{array}{cc}
 z_1 & z_2 \\
 -\overline{z_2} &\overline{ z_1} \\
 \end{array}
  \right)$ where $z_1,z_2\in\mathbb{C}$,
  then $\chi_Y(t)=(t-z_1)(t-\bar z_1)+z_2\bar z_2\ge
  0$ for all real values  of $t$. Consequently,  $Y$ has two equal real
  eigenvalues or two  conjugate complex  eigenvalues.
   Hence, its Jordan canonical  form, $J_Y$,
has to  be of form $aI $  or $\left(
   \begin{array}{cc}
  c &  0\\
 0 & \overline{c} \\
   \end{array}
 \right)$ where $a\in \mathbb{R}$ and $c$
 is a complex with nonzero imaginary  part. Using (\ref{p(t)}) we know
 that, for a solution $Y$ of Equation (\ref{eq:matrix}),  $J_Y$  has to
  be  one of the  following $s+k$ matrices (up to the order of the diagonal entries):
  \begin{equation}\label{J_Y}\xi_1I ,\ldots,\xi_sI ,\left(
  \begin{array}{cc}
 \eta_1 &0 \\
   0  & \bar{\eta_1} \\
  \end{array}
  \right),\ldots,\left(
  \begin{array}{cc}
 \eta_k &0 \\
   0  & \bar{\eta_k} \\
  \end{array}
  \right).\end{equation}
Next we will prove that each of the above cases can occur.

\

{\bf Case 1}:  $J_Y=\xi_iI $ for $i=1,\ldots,s$.

In this case $Y=J_Y$. Clearly,  $Y$  is  a  solution  of Equation
(\ref{eq:matrix})
  iff
  $p_n^\sigma\xi_i^n+\cdots+p_1^\sigma\xi_i+d_0I  =0$, iff  $\xi_i$  is a  common real  root
  of  both  $$t_1^{(n)}x^n+\cdots+t_1^{(1)}x+d_0=0$$  and  $$t_2^{(n)}x^n+\cdots+t_2^{(1)}x=0.$$
  But  $\xi_i$  is  a  real  root  of  $\tilde p(t)$, the
  equalities $t_1^{(n)}\xi_i^n+\cdots+t_1^{(1)}\xi_i+d_0=0$
  and  $t_2^{(n)}\xi_i^n+\cdots+t_2^{(1)}\xi_i=0$  hold  naturally.
  Hence,  $\xi_iI$ $(i=1,\ldots,s)$  is a solution of the  equation
  (\ref{eq:matrix}).

\

  {\bf Case 2}:
  $J_Y=\left(
  \begin{array}{cc}
  \eta_i & 0 \\
  0 & \overline{\eta_i} \\
  \end{array}
  \right)$ for
  $i=1,\ldots,k$.

  We  may  assume  that  (see \cite{ZY})
 $$Y=\left(
  \begin{array}{cc}
  z_1 & z_2 \\
  -\overline{z_2} & \overline{z_1} \\
  \end{array}
   \right)\left(
  \begin{array}{cc}
  \eta_i & 0 \\
  0 & \overline{\eta_i} \\
  \end{array}
  \right)\left(
  \begin{array}{cc}
  z_1 & z_2 \\
  -\overline{z_2} & \overline{z_1} \\
  \end{array}
   \right)^{-1}$$
   for   some   $z_1,   z_2\in \mathbb{C}$  with   $|z_1|^2+   |z_2|^2\neq0$.   Then $Y$
  is  a  solution  of  (\ref{eq:matrix})  if  and  only  if
$$\begin{array}{c}
  p_n^\sigma\left( \begin{array}{cc}  z_1 & z_2 \\  -\overline{z_2} & \overline{z_1} \\
 \end{array}  \right)\left( \begin{array}{cc} \eta_i^{n} & 0 \\ 0 & \overline{\eta_i}^{n} \\
  \end{array}
  \right)\left(
  \begin{array}{cc}
  z_1 & z_2 \\
  -\overline{z_2} & \overline{z_1} \\
  \end{array}
   \right)^{-1}+\cdots \\
  \ \ \ \ \ \ \ \ \ \ \ \ \ +p_1^\sigma\left(
  \begin{array}{cc}
  z_1 & z_2 \\
  -\overline{z_2} & \overline{z_1} \\
  \end{array}
   \right)\left(
  \begin{array}{cc}
  \eta_i^{} & 0 \\
  0 & \overline{\eta_i}^{} \\
  \end{array}
  \right)\left(
  \begin{array}{cc}
  z_1 & z_2 \\
  -\overline{z_2} & \overline{z_1} \\
  \end{array}
   \right)^{-1}+d_0I
   =0,
\end{array}$$
  i.e.,
   \begin{equation}\label{eq:2.4}
  p_n^\sigma\left(
  \begin{array}{cc}
  z_1 & z_2 \\
  -\overline{z_2} & \overline{z_1} \\
  \end{array}
   \right)\left(
  \begin{array}{cc}
  \eta_i^{n} & 0 \\
  0 & \overline{\eta_i}^{n} \\
  \end{array}
  \right)+\cdots \hskip 2cm \end{equation}
$$
\hskip 1cm\begin{array}{cc}
   \ \ \ \ \ & \ \ \ \ \ \ \ +p_1^\sigma\left(
  \begin{array}{cc}
  z_1 & z_2 \\
  -\overline{z_2} & \overline{z_1} \\
  \end{array}
   \right)\left(
  \begin{array}{cc}
  \eta_i^{} & 0 \\
   0  & \overline{\eta_i}^{} \\
  \end{array}
  \right)+d_0\left(
  \begin{array}{cc}
  z_1 & z_2 \\
  -\overline{z_2} & \overline{z_1} \\
  \end{array}
   \right)=0.
  \end{array}
   $$
   In   other   words,   for   any  nonzero   solution $\left(
  \begin{array}{cc}
  z_1 & z_2 \\
  -\overline{z_2} & \overline{z_1} \\
  \end{array}
   \right)$
   of  (\ref{eq:2.4}),   we   obtain  a   solution   for
    (\ref{eq:matrix}) of the following form:
   \begin{equation}\label{eq:2.5}
   \left(
  \begin{array}{cc}
  z_1 & z_2 \\
  -\overline{z_2} & \overline{z_1} \\
  \end{array}
   \right)\left(
  \begin{array}{cc}
  \eta_i^{} & 0 \\
   0  & \overline{\eta_i}^{} \\
  \end{array}
  \right)\left(
  \begin{array}{cc}
  z_1 & z_2 \\
  -\overline{z_2} & \overline{z_1} \\
  \end{array}
   \right)^{-1}.
   \end{equation}

   We   write   (\ref{eq:2.4})   into   two   equations
   $$
   \left\{\begin{array}{cc}
  t_1^{(n)}\eta_i^nz_1+t_2^{(n)}\eta_i^n(-\overline{z_2})
  +\cdots+t_1^{(1)}\eta_iz_1+t_2^{(1)}\eta_i(-\overline{z_2})+d_0z_1=0 &  \\
  t_1^{(n)}\overline{\eta_i}^nz_2+t_2^{(n)}\overline{\eta_i}^n\overline{z_1}+\cdots+t_1^{(1)}
  \overline{\eta_i}z_2+t_2^{(1)}\overline{\eta_i}\overline{z_1}+d_0z_2=0,  \ \ \ \ \ \ \ &
  \end{array}\right.
   $$i.e.,
   \begin{equation}\label{eq:2.6}
  \left\{\begin{array}{cc}
  (t_2^{(n)}\overline{\eta_i}^n+\cdots+t_2^{(1)}\overline{\eta_i})
  \overline{z_1}+(t_1^{(n)}\overline{\eta_i}^n+\cdots+t_1^{(1)}
  \overline{\eta_i}+d_0)z_2=0
   \\ (\overline{t_1^{(n)}}\overline{\eta_i}^n+\cdots
   +\overline{t_1^{(1)}}\overline{\eta_i}+d_0)\overline{z_1}-(\overline{t_2^{(n)}}
   \overline{\eta_i}^n+\cdots+\overline{t_2^{(1)}}\overline{\eta_i})z_2=0.
  \end{array}\right.
   \end{equation}
   Considering
  (\ref{eq:2.6})
   as   a   linear system   in   the  variables
   $X=\left(
   \begin{array}{c}
  \overline{z_1} \\
  z_2 \\
   \end{array}
 \right)$,
 we see that its determinant of  the coefficient matrix  is
 $$D=\left|\begin{array}{cc}
  t_2^{(n)}\overline{\eta_i}^n+\cdots+t_2^{(1)}\overline{\eta_i} & t_1^{(n)}\overline{\eta_i}^n
  +\cdots+t_1^{(1)}\overline{\eta_i}+d_0\\
  \overline{t_1^{(n)}}\overline{\eta_i}^n+\cdots+\overline{t_1^{(1)}}\overline{\eta_i}+d_0 &
  -(\overline{t_2^{(n)}}\overline{\eta_i}^n+\cdots+\overline{t_2^{(1)}}\overline{\eta_i})
  \end{array}\right|.$$
  Since  $\eta_i$  and $\bar\eta_i$  are  roots of
  $$\tilde p(t)= \det P(t)=\left|\begin{array}{cc}
  t_1^{(n)}t^n+\cdots+t_1^{(1)}t+d_0 & t_2^{(n)}t^n+\cdots+t_2^{(1)}t\\
  -(\overline{t_2^{(n)}}t^n+\cdots+\overline{t_2^{(1)}}t) &
  \overline{t_1^{(n)}}t^n+\cdots+\overline{t_1^{(1)}}t+d_0
  \end{array}\right|,$$
  it  follows  that  $D=0$,  which  shows  that  (\ref{eq:2.6})  always  has
  nonzero  solutions.  Since  all  coefficients of  (\ref{eq:2.6})  are  known,  the
  solution  set  of   (\ref{eq:2.6})   in
  $X=\left(
   \begin{array}{c}
  \overline{z_1} \\
  z_2 \\
   \end{array}
 \right)$
 can be given clearly, which will be denoted by $\Gamma_i$.

 If $\eta_i$ and $\overline{\eta_i}$ simultaneously satisfy $t_1^{(n)}x^n+\cdots+t_1^{(1)}x+d_0=0$ and
 $t_2^{(n)}x^n+\cdots+t_2^{(1)}x=0$, then (2.9) becomes trivial,  and $\Gamma_i=\mathbb{C}^{2\times1}$.
 Consequently, any element  of the form  (2.8) is a solution of (2.3).

 Now suppose (2.9) is nontrivial. Then $\Gamma_i$ is of   dimension 1,   and  $\Gamma_i=\left\{z\left(
  \begin{array}{c}
 a^{(i)} \\
 b^{(i)} \\
  \end{array}
  \right)|
  z\in\mathbb{C}\right\},$
  in which $\left(
  \begin{array}{c}
 a^{(i)} \\
 b^{(i)} \\
  \end{array}
  \right)$
  is  a  fixed  nonzero  solution  of
  (\ref{eq:2.6}) with $|a^{(i)}|^2+
 |b^{(i)}|^2=1$.
  Up  to  now,  we  have actually  provided  a  method  to  find  all the  roots  of  (2.1)  in
  $\mathbb{H}$.  To  summarize  our  result  as  a  theorem, we need to introduce some notions.

\vskip 10pt

  Let
  $p_i=t_1^{(i)}+t_2^{(i)}\mathbf{j}\in \mathbb{H}$ for $i=1,\ldots,n$ where
  $t_1^{(i)},t_2^{(i)}\in\mathbb{C}$.
  We call  the  following  four  polynomials  {\it the  derived  polynomials}  of  Equation (2.2):
  $$\begin{array}{cc}
 f_1(t)=t_1^{(n)}t^n+\cdots+t_1^{(1)}t+d_0, & f_2(t)=t_2^{(n)}t^n+\cdots+t_2^{(1)}t ;\\
 \bar f_1(t)=\overline{t_1^{(n)}}t^n+\cdots+\overline{t_1^{(1)}}t+d_0, &
 \bar f_2(t)=\overline{t_2^{(n)}}t^n+\cdots+\overline{t_2^{(1)}}t.
  \end{array}$$
  We  define  the  {\it  discriminant polynomial} of  Equation  (2.2)  as
  $\tilde p(t)=f_1(t)\bar f_1(t)+f_2(t)\bar  f_2(t)$.
  We  factor  it  as  in  (2.5). Remark that these two $\tilde p
 (t)$ are essentially equal.
  Introduce  sets $T_1$ and $T_2$ as follows\\

 $T_1=\{\eta\in\{\eta_1,\ldots, \eta_k\}\,|\,
f_1(\eta)=f_2(\eta)=\bar f_1(\eta)=\bar f_2(\eta)=0\},$\\

$T_2=\{\eta_1,\ldots, \eta_k\}\backslash T_1.$\\

  Now we can state our
   main result

  \begin{theorem}\label{theorem1}
   With the above notations, the solution
  set of (2.2)
  over $\mathbb{H}$ is
  \begin{equation}\label{formula1}
  \{\xi_1,\ldots, \xi_s\}\dot{\cup}_{\eta_i\in
  T_2}\left\{\omega_i\right\}\dot{\cup}_{\eta_i\in
  T_1}[\eta_i],\end{equation}
  where $ \omega_i$ takes
  \begin{equation}\label{zero1}
  \frac{1}{|f_1(\eta_i)|^2+|f_2(\eta_i)|^2}\{|f_2(\eta_i)|^2\eta_i+|f_1(\eta_i)|^2\overline{\eta_i}
  -2f_2(\eta_i)\overline{f_1(\eta_i)}({\rm{Im}}\eta_i)\mathbf{k}\}
\end{equation}
  as its value if $|f_1(\eta_i)|^2+|f_2(\eta_i)|^2\neq 0$,  otherwise takes
\begin{equation}\label{zero2}
       \frac{1}{|f_1(\overline{\eta_i})|^2+|f_2(\overline{\eta_i})|^2} \{|f_1(\overline{\eta_i})|^2\eta_i
       +|f_2(\overline{\eta_i})|^2
  \overline{\eta_i}+2f_2(\overline{\eta_i})\overline{f_1(\overline{\eta_i})}({\rm{Im}}\eta_i)\mathbf{k}\}
  \end{equation}
  as its value,  in which $\Im \eta_i$ means the imaginary  part (real number) of $\eta_i$. Moreover, the
  union of the  first two parts in the above set  is the set of isolated solutions
  and the third part of the above  set is the set of spherical solutions.
 \end{theorem}

 \begin{proof}
 To finish the proof we need to continue the argument on the case when (2.9) is
 nontrivial, i.e., $ \eta_i\in T_2$. For any $ \eta_i\in T_2$,  let
  $\left(
    \begin{array}{c}
      a^{(i)} \\
      b^{(i)}\\
    \end{array}
  \right)$ be a nonzero complex solution of the  system $\left(
    \begin{array}{cc}
  f_2(\overline{\eta_i}) &  f_1(\overline{\eta_i}) \\
   \bar f_1(\overline{\eta_i}) &  -\bar f_2(\overline{\eta_i}) \\
    \end{array}
  \right)X=0$  with $|a^{(i)}|^2+
    |b^{(i)}|^2=1$.
 Since the set of nonzero solutions of (\ref{eq:2.6}) in unknown
 $X=\left(
   \begin{array}{c}
     \overline{z_1} \\
     z_2 \\
   \end{array}
 \right)$
 is
 $\left\{l\left(
   \begin{array}{c}
     a^{(i)} \\
     b^{(i)}\\
   \end{array}
     \right)\mid l\neq
     0,
     l\in\mathbb{C}\right\}$,
     the     solutions of     (\ref{eq:matrix})     corresponding     to     the     Jordan
     canonical     form $\left(
   \begin{array}{cc}
     \eta_i & 0 \\
     0 & \overline{\eta_i }\\
   \end{array}
     \right)$
     are
      $$  \left(
      \begin{array}{cc}
        \overline{la^{(i)}} & lb^{(i)} \\
        -\overline{lb^{(i)}} & la^{(i)} \\
      \end{array}
    \right)\left(
 \begin{array}{cc}
   \eta_i & 0 \\
  0  & \overline{\eta_i}\\
 \end{array}
       \right)\left(
      \begin{array}{cc}
        \overline{la^{(i)}} & lb^{(i)} \\
        -\overline{lb^{(i)}} & la^{(i)} \\
      \end{array}
    \right)^{-1}$$
    $$ = \left(
      \begin{array}{cc}
        \overline{a^{(i)}} & b^{(i)} \\
        -\overline{b^{(i)}} & a^{(i)} \\
      \end{array}
    \right)\left(
 \begin{array}{cc}
   \eta_i & 0 \\
  0  & \overline{\eta_i}\\
 \end{array}
       \right)\left(
      \begin{array}{cc}
        \overline{a^{(i)}} & b^{(i)} \\
        -\overline{b^{(i)}} & a^{(i)} \\
      \end{array}
    \right)^{-1},$$
    which    is    actually    one    value,    and    in    $\mathbb{H}$    which    can    be    written
 as  $$(\overline{a^{(i)}}+b^{(i)}\mathbf{j})\eta_i(\overline{a^{(i)}}+b^{(i)}\mathbf{j})^{-1}
 =(\overline{a^{(i)}}+b^{(i)}\mathbf{j})\eta_i({a^{(i)}}-b^{(i)}\mathbf{j})$$
  $$={|a^{(i)}|^2\eta_i+|b^{(i)}|^2\overline{\eta_i}
  +\overline{a^{(i)}}b^{(i)}(\overline{\eta_i}-\eta_i)\mathbf{j}},$$
  which can not be a real number.
  Therefore, the solution set of
  $p_nx^n+\cdots+p_1x+d_0=0$ in the
  skew-field $\mathbb{H}$ is
  \begin{equation}\label{formula2b}
   \{\xi_1,\ldots, \xi_s\}\dot{\cup}_{\eta_i\in
  T_2}\left\{{|a^{(i)}|^2\eta_i+|b^{(i)}|^2\overline{\eta_i}-2\overline{a^{(i)}}b^{(i)}(\Im \eta_i)
  \mathbf{k}}\right\}\dot{\cup}_{\eta_i\in
  T_1}[\eta_i].
  \end{equation}
  Note that $|f_1(\eta_i)|^2+|f_2(\eta_i)|^2$ and $|f_1(\overline{\eta_i})|^2+|f_2(\overline{\eta_i})|^2$
  can not be $0$ simultaneously for
  $\eta_i\in T_2$, then we can take   $$\left(
    \begin{array}{c}
      a^{(i)} \\
      b^{(i)}\\
    \end{array}
  \right)=\left(
    \begin{array}{c}
      \frac{\overline{f}_2(\overline{\eta_i})}{\sqrt{|\overline{f}_1
      (\overline{\eta_i})|^2+|\overline{f}_2(\overline{\eta_i})|^2}} \\
      \frac{\overline{f}_1(\overline{\eta_i})}{\sqrt{|\overline{f}_1
      (\overline{\eta_i})|^2+|\overline{f}_2(\overline{\eta_i})|^2}}\\
    \end{array}
  \right)$$ for the case $|f_1(\eta_i)|^2+|f_2(\eta_i)|^2\neq
  0$, otherwise we take
   $$\left(
    \begin{array}{c}
      a^{(i)} \\
      b^{(i)}\\
    \end{array}
  \right)=\left(
    \begin{array}{c}
      \frac{{f}_1(\overline{\eta_i})}{\sqrt{|{f}_1(\overline{\eta_i})|^2+|{f}_2(\overline{\eta_i})|^2}} \\
      \frac{-{f}_2(\overline{\eta_i})}{\sqrt{|{f}_1(\overline{\eta_i})|^2+|{f}_2(\overline{\eta_i})|^2}}\\
    \end{array}
  \right).$$
  After  manipulations,  the set in (\ref{formula2b}) becomes
   $$
  \{\xi_1,\ldots, \xi_s\}\dot{\cup}_{\eta_i\in
  T_2}\left\{\omega_i\right\}\dot{\cup}_{\eta_i\in
  T_1}[\eta_i],$$
  where $ \omega_i$ takes
  $$\frac{1}{|f_1(\eta_i)|^2+|f_2(\eta_i)|^2}\{|f_2(\eta_i)|^2\eta_i+|f_1(\eta_i)|^2\overline{\eta_i}
  -2f_2(\eta_i)\overline{f_1(\eta_i)}({\rm{Im}}\eta_i)\mathbf{k}\}$$ as its value if $|f_1(\eta_i)|^2+|f_2(\eta_i)|^2\neq0$;
  otherwise it takes
      $$ \frac{1}{|f_1(\overline{\eta_i})|^2+|f_2(\overline{\eta_i})|^2}
      \{|f_1(\overline{\eta_i})|^2\eta_i+|f_2(\overline{\eta_i})|^2
  \overline{\eta_i}+2f_2(\overline{\eta_i})\overline{f_1(\overline{\eta_i})}({\rm{Im}}\eta_i)\mathbf{k}\}$$
  as its value.

 Finally, it is clear that  $[\eta_i]$ contains no real
  numbers for $\eta_i\in T_1$.
 This
  completes the proof.
\end{proof}
Note that there is no repetition in the solution set given in
(\ref{formula1}) and one can use (\ref{zero1}) or (\ref{zero2}) for
$\omega_i$ if
$$(|f_1(\eta_i)|^2+|f_2(\eta_i)|^2)(|f_1(\eta_i)|^2+|f_2(\eta_i)|^2)\ne0.$$

Theorem \ref{theorem1} shows, once we get a decomposition with the
form (\ref{p(t)}) of the discriminant polynomial $\tilde p(t)$, then
we can produce all roots of the  quaternionic equation
(\ref{eq:simple 1}) by formula (\ref{formula1}).

 From Theorem 1 and the arguments before Theorem 1, we can easily see the following well-known results.
 \begin{corollary}\label{corollary2}
  \begin{enumerate}
\item[(a).] Any simple quaternionic equation  $q_nx^n+\cdots+q_1x+q_0=0(q_n\neq
 0)$ has a root in $\mathbb{H}$.
\item[(b).] The simple quaternionic equation  $q_nx^n+\cdots+q_1x+q_0=0(q_n\neq
 0)$ has a finite number of roots  in
 $\mathbb{H}$ iff it has at most
   $n$ distinct roots in $\mathbb{H}$.
\item[(c).] The roots of  $q_nx^n+\cdots+q_1x+q_0=0(q_n\neq
 0)$ are distributed in at most
 $n$ conjugate classes, and there are at most $n$ real roots among them.
\end{enumerate}\end{corollary}

  \begin{proof} (a) is obvious.

  (b) From Theorem 1 we need to only show
  that $[\eta_i]$ is an infinite set if $\eta_i\in T_1$.
  It is well-known from \cite{ZFZ} that $u_1,u_2\in \mathbb{H} $ are
  conjugate (i.e., there exists nonzero $ q\in\mathbb{H}$
  such that $
  u_1=qu_2q^{-1}$) iff $\Re u_1=\Re u_2$ and
  $|u_1|=|u_2|$. Since $\eta_i\in T_1$ has a
 nonzero imaginary part, $[\eta_i]$ is an infinite set. Thus $q_nx^n+\cdots+q_1x+q_0=0$ has a finite
  number of roots  in $\mathbb{H}$ iff $T_1=\varnothing$, iff $p_nx^n+\cdots+p_1x+d_0=0$ has
  at most $n$ roots in $\mathbb{H}$ since $s+k\le n$ (See (\ref{p(t)}) for the notations).

  (c) follows from $s+k\le n$.
  \end{proof}

Now we give a quick method to solve simple quaternionic polynomials
with all real coefficients or with all complex coefficients. These
results do not seem to be deduced easily from the results in
\cite{JO1}.

\begin{corollary}\label{corollary3}
 \begin{enumerate}
\item[(a).] If all $q_i$ in  $q_nx^n+\cdots+q_1x+q_0=0 (q_n\neq
 0)$ are real numbers and
  the solution set of this
 equation in
 $\mathbb{C}$ is $\{\xi_1,\ldots, \xi_s,  \zeta_1,
 \overline{\zeta_1},$ $\ldots, \zeta_t, \overline{\zeta_t}
 \}$, where $\xi_1,\ldots,\xi_s$ are  distinct real
numbers, $ \zeta_1, \ldots,\zeta_t$ are
 distinct  nonreal complex numbers,
  then the solution set of this equation in
 $\mathbb{H}$ is $$\{\xi_1,\ldots, \xi_s,\}\cup[\zeta_1]\cup\cdots\cup[\zeta_t].$$
\item[(b).] More generally, if all $q_i$ in  $q_nx^n+\cdots+q_1x+q_0=0(q_n\neq
 0)$ are complex numbers, and
  the solution set of this
 equation in
 $\mathbb{C}$ is $\{\xi_1,\ldots, \xi_s, \eta_1,\ldots, \eta_k, \zeta_1,
 \overline{\zeta_1},\ldots, \zeta_t, \overline{\zeta_t}
 \}$, where $\xi_1,\ldots,\xi_s$ are  distinct real
numbers, $\eta_1,\ldots, \eta_k, \zeta_1, \ldots,\zeta_t$ are
 distinct  nonreal complex numbers (each
$\overline{\eta_i}$ is no longer the root of this equation),
  then the solution set of this equation in
 $\mathbb{H}$ is $$\{\xi_1,\ldots, \xi_s, \eta_1,\ldots,
 \eta_k,\}\cup[\zeta_1]\cup\cdots\cup[\zeta_t].$$
\item[(c).] Let $f_1(t)$, $f_2(t)$, $\bar f_1(t)$ and $\bar f_2(t)$ be the derived polynomials for
(\ref{eq:simple 1}). Then (\ref{eq:simple 1}) has  finitely many
solutions iff the complex polynomial $\gcd (f_1(t), f_2(t),\bar
f_1(t),\bar f_2(t))$ has no nonreal  complex root  iff the complex
polynomial $\gcd (f_1(t), f_2(t))$ has no nonreal conjugate complex
roots.
\end{enumerate}\end{corollary}

\begin{proof} (a)  This is a special case of Part (b).

  (b) When  the equation considered has only complex
  coefficients, two of the derived polynomials are $f_1(t)=q_nt^n+\cdots+q_1t+q_0$ (up to a
  complex
  scalar), and $f_2(t)=0$. So, the roots of the
  discriminant polynomial are $\xi_1,\ldots, \xi_s, \eta_1,
  \overline{\eta_1},\ldots, \eta_k,\overline{\eta_k},
  \zeta_1, \overline{\zeta_1},\ldots, \zeta_t,
  \overline{\zeta_t}$. For each $\eta_i \, (i=1,\ldots,k)$,
  since $\overline{f}_1(\eta_i)\neq 0$,  $\eta_i$ is in
  $T_2$.
  We can take $\left(
 \begin{array}{c}
 a^{(i)} \\
 b^{(i)} \\
 \end{array}
  \right)=\left(
 \begin{array}{c}
 1 \\
 0 \\
 \end{array}
  \right)$ as the unit complex solution of the  system $\left(
  \begin{array}{cc}
  f_2(\overline{\eta_i}) &  f_1(\overline{\eta_i}) \\
   \bar f_1(\overline{\eta_i}) &  -\bar f_2(\overline{\eta_i}) \\
  \end{array}
   \right)X=0$. Then, $|a^{(i)}|^2\eta_i+|b^{(i)}|^2\overline{\eta_i}-2b^{(i)}\overline{a^{(i)}}(\Im
   \eta_i)\mathbf{k}=\eta_i$. It is easy to see  that $\zeta_i\in T_1$. This completes the proof.

   (c) Suppose $c$ is a nonreal complex root of gcd$(f_1(t), f_2(t),\bar f_1(t),\bar
   f_2(t))$. Then both $c$ and $\bar c$ are roots of the
   discriminant polynomial,
   $f_1(c)=f_2(c)=\bar f_1(c)=\bar f_2(c)=0$ and $f_1(\bar c)=f_2(\bar c)=\bar f_1(\bar c)=\bar f_2(\bar
   c)=0$, which implies that at least one of $c$, $\bar c$ is in $T_1$. Thus, $T_1=\varnothing$ iff
   gcd$(f_1(t), f_2(t),\bar f_1(t),\bar
   f_2(t))$ has no nonreal complex root. From Theorem 1 we see  that (\ref{eq:simple 1}) has finitely many solutions iff
   $T_1=\varnothing$. The conclusions in the corollary
   follow easily.
 \end{proof}

Now we can give a simplified version of  Theorem 1.

\begin{theorem} Consider the simple quaternionic equation
$p(x):=p_nx^n+\cdots+p_1x+d_0=0,$ where $p_i\in\mathbb{H}$ with $
p_n\neq 0$ and $d_0=0 \ {\text{or}} \ 1$. We write
$p(t)=g(t)(g_1(t)+g_2(t)\mathbf{j})$ where  $t$ is considered as a
real variable, $g,g_1,g_2\in\C[t]$ with $\gcd(g_1,g_2)=1$. Let the
complex solution sets for $g(t)$ and $\tilde g(t)=g_1(t)\bar
g_1(t)+g_2(t)\bar g_2(t)$ are
\begin{equation}\label{first}\{\xi_1,\ldots,\xi_{s};
\lambda_1,\bar \lambda_1,\lambda_2, \bar\lambda_2,\cdots,
\lambda_{t},\bar\lambda_{t}; \eta_1,\ldots, \eta_{k_1}
\},\end{equation}
\begin{equation}\label{sec}\{\eta_{k_1+1},\bar\eta_{k_1+1}\ldots,
\eta_{k},\bar\eta_k \}\end{equation} respectively. We may assume
that
 $\xi_1,\ldots,\xi_s$ are
distinct real numbers; $\eta_1,\ldots,$ $
\eta_{k_1},\eta_{k_1+1},\bar\eta_{k_1+1}\ldots, \eta_{k},\bar\eta_k
, \lambda_1,\bar \lambda_1,\lambda_2, \bar\lambda_2,\cdots,
\lambda_{t},\bar\lambda_{t}$ are distinct nonreal complex numbers
after deleting those   $\eta_i$'s and $\bar\eta_i$'s in (\ref{sec})
if one of them appears in the set (\ref{first}). Then the
quaternionic solution set of $p(x)$ is
  \begin{equation}\label{formula2}
  \{\xi_1,\ldots, \xi_s,; \omega_1,\cdots,  \omega_{k}\}\dot{\cup}_{ i=1
   }^{t}[\lambda_i] ,\end{equation}
where $$\omega_i= \frac{1}{|g_1(\bar\eta_i)|^2+|
g_2(\bar\eta_i)|^2}\{|
g_2(\bar\eta_i)|^2\bar\eta_i+|g_1(\bar\eta_i)|^2{\eta_i}
  +2 g_2(\bar\eta_i)\overline{g_1(\bar\eta_i)}({\rm{Im}}\eta_i)\mathbf{k}\}.
 $$
 \end{theorem}

 \begin{proof} We continue to use the notations in Theorem 1. We know that $f_1(t)=g(t)g_1(t)$,
 $f_2(t)=g(t)g_2(t)$, $\tilde p(t)=g(t)\bar g(t)\tilde g(t)$,  and $\tilde g(t)$ has no real root.

From Theorem 1 we see that $\{\xi_1,\ldots, \xi_s\}\dot{\cup}_{ i=1
}^{t_1}[\lambda_i]$ are zeros of $p(x)$. All other zeros come from
$\{\eta_1,\ldots, \eta_{k_1},\eta_{k_1+1},\bar\eta_{k_1+1}\ldots,
\eta_{k},\bar\eta_k\}$. For each $\eta_i$, we see that
$|f_1(\bar\eta_i)|^2+|f_2(\bar\eta_i)|^2\neq0$. Then using
(\ref{zero2})
 in Theorem 1 and simplifying we obtain $\omega_i$. This
completes the proof.
\end{proof}

Remark that the above theorem simplifies the computation for finding
all zeros of $p(x)$, and the following known result  (see \cite{PS})
can follow easily from the above theorem.

\begin{corollary}
The  spherical zeros of  simple quaternionic polynomial
$p(x):=p_nx^n+\cdots+p_1x+p_0( p_i\in\mathbb{H}, p_n\neq 0)$ are
distributed in at most
 {\rm{INT}}$(\frac{n}{2})$ conjugate classes, where {\rm{INT}}$(\frac{n}{2})$ means
  the   integral function value at $\frac{n}{2}$.
\end{corollary}

\section{Algorithm}
Based on our Theorem 1, we now can give  an algorithm to solve the
quaternionic equation $q_nx^n+\cdots+q_1x+q_0=0 \, (q_n\neq 0)$, as
follows.

\

$\emph{Algorithm 1 }$ (for solving the simple quaternionic
polynomial equation $q_nx^n+\cdots+q_1x+q_0=0$)

{\bf{Step 1.}}\  \ Write the equation as $p_nx^n+\cdots+p_1x+d_0=0$
with $d_0=0$ or $1$ (in fact, if $q_0\neq 0$, simply multiply
 the equation  by $q_0^{-1}$ on the left). Write
$p_i=t_1^{(i)}+t_2^{(i)}\mathbf{j}$ for $i=1,\ldots, n$ with $
t_1^{(i)}, t_2^{(i)}\in\mathbb{C}$. Find the derived polynomials and
discriminant polynomial of $p_nx^n+\cdots+p_1x+d_0=0$:
$$\begin{array}{cc}
 f_1(t)=t_1^{(n)}t^n+\cdots+t_1^{(1)}t+d_0, & f_2(t)=t_2^{(n)}t^n+\cdots+t_2^{(1)}t ;\\
 \bar f_1(t)=\overline{t_1^{(n)}}t^n+\cdots+\overline{t_1^{(1)}}t+d_0, &
 \bar f_2(t)=\overline{t_2^{(n)}}t^n+\cdots+\overline{t_2^{(1)}}t,
  \end{array}$$
$$\tilde p(t)=f_1(t)\bar f_1(t)+f_2(t)\bar f_2(t).\ \ \ \ \ \ \ \ \ \ \ \ \ \ \
\ \ \ \ \ \ \ \ \ \ \ \ \ \ \ \ \ \ \ \ \ \ \ $$ Make sure the
coefficients of $\tilde p(t)$ are real.

  {\bf{Step  2.}}\
  \ Compute all distinct zeros (real or complex) of the discriminant polynomial $\tilde p(t)$
  (in MATLAB, use the command roots). Denote
these zeros by $\xi_1,\ldots,\xi_s$, $\eta_1,\ldots$, $\eta_k$,
$\overline{\eta_1},\ldots$, $\overline{\eta}_k$ such that
   $\xi_1,\ldots,\xi_s$ are distinct real
numbers and $\eta_1,\ldots$, $\eta_k$ are distinct nonreal complex
numbers.
 Then classify $\eta_1,\ldots,\eta_k$
into two sets $T_1$ and $T_2$:

 $T_1=\{ \eta\in\{\eta_1,\ldots,
\eta_k\}\,|\, f_1(\eta)=f_2(\eta)=\bar f_1(\eta)=\bar
f_2(\eta)=0\}$,

$T_2=\{\eta_1,\ldots, \eta_k\}\backslash T_1$.\\
When the simple quaternionic polynomial considered is a polynomial
with real coefficients, then directly set $T_1= \{\eta_1,\ldots,
\eta_k\}$ and  $T_2=\varnothing$.

  {\bf{Step  3.}}\
  \
   For each $\eta_i\in T_2$, compute $f_1(\eta_i)$, $f_2(\eta_i)$. Then by Formula (\ref{formula1}), output all roots of
  $q_nx^n+\cdots+q_1x+q_0=0$.\\

If we use our Theorem 4, then we  get  a better version of
Algorithm
1. \\

$\emph{Algorithm 1' }$ (for solving the simple quaternionic
polynomial equation $q_nx^n+\cdots+q_1x+q_0=0$)

{\bf{Step 1.}}\  \ Write the equation as
$p(x):=p_nx^n+\cdots+p_1x+d_0=0$ with $d_0=0$ or $1$ (in fact, if
$q_0\neq 0$, simply multiply
 the equation  by $q_0^{-1}$ on the left).
Then write $p(t)=g(t)(g_1(t)+g_2(t)\mathbf{j})$ where  $t$ is
considered as a real variable, $g,g_1,g_2\in\C[t]$ with
$\gcd(g_1,g_2)=1$. Now we compute $\tilde g(t)=g_1(t)\bar
g_1(t)+g_2(t)\bar g_2(t)$.

\

{\bf{Step  2.}}
 Compute all distinct zeros (real or complex) for $g(t)$ and  $\tilde g(t)$ respectively (in MATLAB, use the command
 roots):
\begin{equation}\label{first2}\{\xi_1,\ldots,\xi_{s};
\lambda_1,\bar \lambda_1,\lambda_2, \bar\lambda_2,\cdots,
\lambda_{t},\bar\lambda_{t};\eta_1,\ldots, \eta_{k_1}
\},\end{equation}
\begin{equation}\label{sec2}\{\eta_{k_1+1},\bar\eta_{k_1+1}\ldots,
\eta_{k},\bar\eta_k \}, \end{equation} where $\xi_1,\ldots,\xi_s$
are distinct real numbers; $\{ \lambda_1,\bar \lambda_1,\lambda_2,
\bar\lambda_2,\cdots, \lambda_{t},\bar\lambda_{t}\}$
 and $\{\eta_1,\ldots, \eta_{k_1},\eta_{k_1+1},\bar\eta_{k_1+1}\ldots,
\eta_{k},\bar\eta_k\}$ are two sets of distinct nonreal complex
numbers. Delete those   $\eta_i$'s and $\bar\eta_i$'s in
(\ref{sec2}) if one of them appears in the set (\ref{first2}).

\

{\bf{Step 3.}}
   For each $\eta_i$, we compute
   \begin{equation}\label{formulaomega}
   \omega_i= \frac{1}{|g_1(\bar\eta_i)|^2+|
g_2(\bar\eta_i)|^2}\{|
g_2(\bar\eta_i)|^2\bar\eta_i+|g_1(\bar\eta_i)|^2{\eta_i}
  +2 g_2(\bar\eta_i)\overline{g_1(\bar\eta_i)}({\rm{Im}}\eta_i)\mathbf{k}\}.
 \end{equation}
Then the quaternionic solution set of $p(x)$ is
  \begin{equation}\label{formula21}
  \{\xi_1,\ldots, \xi_s,; \omega_1,\cdots,  \omega_{k}\}\dot{\cup}_{ i=1
   }^{t}[\lambda_i] .\end{equation}\\

Prior to our method,  Janovsk$\acute{a}$ and  Opfer  gave an
algorithm in \cite{JO1} for solving the same quaternionis equation
$q_nx^n+\cdots+q_1x+q_0=0 \, (q_n\neq 0)$. In order to introduce
their algorithm precisely, let us first recall the known concept,
companion polynomial.
 For $$p(x)=\sum_{j=0}^nq_jx^j, q_j\in\mathbb{H}, j=0,\ldots, n,\
 q_0,q_n\neq 0,$$
 following Niven \cite{N}, or more recently Janovsk$\acute{a}$ and  Opfer \cite{JO1}, its
 {\it companion polynomial} is defined by
 \begin{equation}\label{companion}
 q_{2n}(x)=\sum_{j,k=0}^n\bar q_jq_kx^{j+k}=\sum_{k=0}^{2n}b_kx^k,
 \textrm{where}\
 b_k=\sum_{j=max(0,k-n)}^{min(k,n)}\bar q_jq_{k-j}\in\mathbb{R}.
 \end{equation}
 We remark that $q_{2n}$ is equal to the discriminant $\tilde p(t)$
 in our case.
 From Pogorui and Shapiro\cite{PS}, we know that all powers $x^j$, $j\in \mathbb{Z}$ of a
 quaternion $x$ have the form $x^j = \alpha x + \beta$ with real $\alpha$, $\beta$.
  In particular, $x^2=2({\rm{Re}} x)\ x-|x|^2$. In order to determine the numbers $\alpha$, $\beta$,
  Janovsk$\acute{a}$ and  Opfer in \cite{JO1} set up the following
  iteration:
  \begin{equation}\label{eq: iteration1}
\left\{\begin{array}{c}
        x^j=\alpha_jx+\beta_j,\ \alpha_j,\beta_j\in\mathbb{R},j=0,1,\ldots, \\
        \alpha_0=0, \beta_0=1, \\
       \alpha_{j+1}=2Rex\ \alpha_j+\beta_j, \\
        \beta_{j+1}=-|x|^2\alpha_j, j=0,1,\dots.
      \end{array}\right.
  \end{equation}
  Now by means of the first line of Iteration (\ref{eq:
  iteration1}), the polynomial $p(x)$ can be rewritten as

$$p(x)=\sum_{j=0}^nq_j(\alpha_jx+\beta_j)=
\left(\sum_{j=0}^nq_j\alpha_j\right)x+\left(\sum_{j=0}^nq_j\beta_j\right)\equiv
A(x)x+B(x),$$ where
\begin{equation}\label{formula3}
A(x)=\sum_{j=0}^nq_j\alpha_j,\  B(x)=\sum_{j=0}^nq_j\beta_j.
\end{equation}
With these in hand, we can state the algorithm given in \cite{JO1},
 as follows.\\

$\emph{Algorithm 2}$ (for solving the simple quaternionic equation
$q_nx^n+\cdots+q_1x+q_0=0$)

\

{\bf{Step 1.}}\  \ Write $q_nx^n+\cdots+q_1x+q_0=0$
 as $p(x)=a_nx^n+\cdots+a_1x+a_0=0$ with $a_n=1$. For this $p(x)$, compute
 the real coefficients $b_0$, $b_1$, $\ldots$, $b_{2n}$ of the companion polynomial
 $q_{2n}(x)$ by formula (\ref{companion}).\\

 {\bf{Step 2.}}\  \ Compute all $2n$ (real and complex) zeros of
 $q_{2n}(x)$,
  denote these zeros by $z_1$, $z_2$, $\ldots$, $z_{2n}$ and order them (if
 necessary) such that $z_{2j-1} = \overline{z_{2j}}$, $j = 1$, $2$, $\ldots$,
 $n$.\\

{\bf{Step 3.}}\  \ Define an integer vector { $\mathbf{ind}$} (like
{\it indicator}) of length $n$, and set all components to zero.
Define a quaternionic vector $Z$ of length $n$, and set all
components to zero.

$\mathbf{For}$ $\mathbf{j:=1:n}$ $\mathbf{do}$

(a) $\mathbf{Put}$ $z:=z_{2j-1}$.

(b) $\mathbf{if}$ $z$ is real, $Z(j):= z$; go to the next step;
$\mathbf{end}$ $\mathbf{if}$

(c) $\mathbf{Compute}$ $v := \overline{A(z)}B(z)$ by formula
(\ref{formula3}), with the help of Iteration (\ref{eq: iteration1}).

(d) $\mathbf{if}$ $v = 0$, put $\mathbf{ind}(j) := 1; Z(j) := z$; go
to the next step; $\mathbf{end}$ $\mathbf{if}$

(e) $\mathbf{if}$ $v\neq 0$, let $(v_1, v_2, v_3, v_4):= v$. Compute
$|w| := \sqrt{ v_2^2+ v_3^2+ v_4^2}$, and put

\begin{equation}\label{formula4}
Z(j) :=\left(\textrm{Re}z, -\frac{|\textrm{Im}z|}{|w|}v_2,
-\frac{|\textrm{Im}z|}{|w|}v_3,
-\frac{|\textrm{Im}z|}{|w|}v_4\right).
\end{equation}
$\mathbf{end }$ $\mathbf{if}$

$\mathbf{end}$ $\mathbf{for}$

\

In this algorithm,  corresponding to a real $z$ the expression
$Z(j)$ produces a real isolated zero $z$, corresponding to ``$v=0$"
the expression $Z(j)$ produces a spherical zero $[z]$, and
corresponding to ``$v\neq 0$" the expression $Z(j)$ produces an
isolated zero $\textrm{Re}z -\frac{|\textrm{Im}z|}{|w|}v_2\mathbf{i}
-\frac{|\textrm{Im}z|}{|w|}v_3\mathbf{j}
-\frac{|\textrm{Im}z|}{|w|}v_4\mathbf{k}$. The output results of
$Z(j)$ produce all zeros of polynomial $q_nx^n+\cdots+q_1x+q_0$.
Algorithm 2's original edition is Section 7 of \cite{JO1}, where
$|w| := \sqrt{ v_2^2+ v_3^3+ v_4^2}$ is false, which is a misprint.

\section{Examples}
 $\textbf{ Example 1}$ \ \ In $\mathbb{H}$, solve the equation
 $p(x):=\mathbf{i}x^3+\mathbf{j}x^2+\mathbf{k}x+1=0$.\\

  $\textit{ Solution 1.}$ Use our Algorithm 1 to do this.  Write the coefficients
   $\mathbf{i},\mathbf{j},\mathbf{k}$ into the form:
   $\mathbf{i}=\mathbf{i}+0\mathbf{j}, \ \mathbf{j}=0+1\mathbf{j},\  \mathbf{k}=0+\mathbf{i}\mathbf{j}$.
   Then the
   derived polynomials of this equation are
   $$\begin{array}{cc}
     f_1(t)=\mathbf{i}t^3+1, & f_2(t)=t^2+\mathbf{i}t, \\
     \bar f_1(t)=-\mathbf{i}t^3+1, & \bar f_2(t)=t^2-\mathbf{i}t,
   \end{array}$$
    the discriminant polynomial is
  $$\tilde p(t)=f_1\bar f_1+f_2\bar f_2=(t-\mathbf{i})(t+\mathbf{i})(t-e^{\mathbf{i}\frac{\pi}{4}})
  (t-e^{\mathbf{i}\frac{3\pi}{4}})(t-e^{\mathbf{i}\frac{5\pi}{4}})(t-e^{i\frac{7\pi}{4}}),$$
  which has no real root. It is easy to see that  $$T_1=\varnothing  \
  \textrm{and}\
  T_2=\{\eta_1= \mathbf{i},\eta_2=e^{\mathbf{i}\frac{\pi}{4}},
  \eta_3=e^{\mathbf{i}\frac{3\pi}{4}}\}.$$

  For $\eta_1= \mathbf{i}$, then $f_1(\mathbf{i})=2\neq 0$ and $f_2(\mathbf{i})=-2$, and we get an
  isolated zero by Formula (\ref{formula1}):
  $\frac{1}{8}\cdot (4\mathbf{i}-4\mathbf{i}-2\cdot(-2)\cdot 2\cdot 1\cdot
  \mathbf{k})=\mathbf{k}$;

  For
  $\eta_2=e^{\mathbf{i}\frac{\pi}{4}}=\frac{\sqrt{2}}{2}+\frac{\sqrt{2}}{2}\mathbf{i}$,
  then
  $f_1(\eta_2)=\frac{2-\sqrt{2}}{2}-\frac{\sqrt{2}}{2}\mathbf{i}$ and
  $f_2(\eta_2)=-\frac{\sqrt{2}}{2}+\frac{\sqrt{2}+2}{2}\mathbf{i}$,
 and  we get another isolated zero by Formula (\ref{formula1}):
  $\frac{\sqrt{2}}{2}+\frac{1}{2}\mathbf{i}+\frac{1}{2}\mathbf{k}$;

  Similarly for $\eta_3=e^{\mathbf{i}\frac{3\pi}{4}}$ we get the
  isolated zero:
  $\frac{\sqrt{2}}{2}+\frac{1}{2}\mathbf{i}+\frac{1}{2}\mathbf{k}$.

 Thus,  the solution set  is
 $\left\{\mathbf{k},
 \frac{\sqrt{2}}{2}+\frac{1}{2}\mathbf{i}+\frac{1}{2}\mathbf{k},
 -\frac{\sqrt{2}}{2}+\frac{1}{2}\mathbf{i}+\frac{1}{2}\mathbf{k}
  \right\}$.\\

   $\textit{ Solution 2.}$ Use our Algorithm 2 to do this.  Write
   $p(t)=(\mathbf{i}t^3+1)+(t^2+\mathbf{i}t)\mathbf{j}:=f_1+f_2\mathbf{j}$,
   then

   $g=\gcd(f_1, f_2)=t+\mathbf{i}$, $g_1=\mathbf{i}t^2+t-\mathbf{i}$,
   $g_2=t$,
   $\tilde g=g_1\bar g_1+g_2\bar g_2=t^4+1$,

   Compute the zeros of $g$: $\eta_1=-\mathbf{i}$.

   Compute the zeros of $\tilde g$:
   $\eta_2=e^{\mathbf{i}\frac{\pi}{4}}$, $\overline{\eta_2}$,
   $\eta_3=e^{\mathbf{i}\frac{3\pi}{4}}$, $\overline{\eta_3}$.

 Now for $\eta_i$($i=1,2,3$), compute $g_1(\overline{\eta_i})$, $|g_1(\overline{\eta_i})|^2$,
 $g_2(\overline{\eta_i})$, and  $|g_2(\overline{\eta_i})|^2$,  by formula (\ref{formulaomega}) we get
 the solution set of
 $\mathbf{i}x^3+\mathbf{j}x^2+\mathbf{k}x+1=0$:
 $$\left\{\mathbf{k},
 \frac{\sqrt{2}}{2}+\frac{1}{2}\mathbf{i}+\frac{1}{2}\mathbf{k},
 -\frac{\sqrt{2}}{2}+\frac{1}{2}\mathbf{i}+\frac{1}{2}\mathbf{k}
  \right\}.$$\\

 $\textbf{ Example 2}$ \ \ Solve the equation
 $x^3+x^2+x+1=0$ in $\mathbb{H}$.\\

   $\textit{ Solution 1.}$ We will use our method to do this first. Since $x^3+x^2+x+1=(x+1)(x+\mathbf{i})(x-\mathbf{i})$,
   from  Corollary \ref{corollary3}(a), we directly know the solution set is
   $\{-1\}\cup[\mathbf{i}]$.\\

$\textit{ Solution 2.}$ Now we use the  method in \cite{JO1} to do
 this.

 First Step. By formula (\ref{companion}), we get

 $q_6(x)=(x^3+x^2+x+1)^2=x^6+2x^5+3x^4+4x^3+3x^2+2x+1.$

 Second Step.
 Compute all zeros of $q_6(x)$:
 $-1, -1, \mathbf{i}, \mathbf{i}, -\mathbf{i}, -\mathbf{i}.$

 Third Step.  For $-1$ we get a real isolated zero $-1$. For $\mathbf{i}$ we need give
  the following expansion by Iteration (\ref{eq: iteration1}):
$$\mathbf{i}^0=0\mathbf{i}+1,\mathbf{i}^1=1\mathbf{i}+0,\mathbf{i}^2=0\mathbf{i}+(-1),\mathbf{i}^3=(-1)\mathbf{i}+0.$$
 Then by formula (\ref{formula3})
get $$A(\mathbf{i})=0+1+0+(-1)=0, B(\mathbf{i})=1+0+(-1)+0=0,$$ and
$v:=\overline{A(\mathbf{i})}B(\mathbf{i})=0.$ It leads to a
spherical zero $[\mathbf{i}]$. Finally, for $-\mathbf{i}$ we have to
repeat the same process as that made for $\mathbf{i}$:
$$(-\mathbf{i})^0=0(-\mathbf{i})+1, (-\mathbf{i})^1=1(-\mathbf{i})+0,$$ $$ (-\mathbf{i})^2=0(-\mathbf{i})+(-1),
(-\mathbf{i})^3=-1(-\mathbf{i})+0,$$
$$A(-\mathbf{i})=0+1+0+(-1)=0, B(-\mathbf{i})=1+0+(-1)+0=0,$$ and
$v:=\overline{A(-\mathbf{i})}B(-\mathbf{i})=0$. It also produces a
spherical zero $[-\mathbf{i}]$. Note that
$[-\mathbf{i}]=[\mathbf{i}]$, so the solution set is
$\{-1\}\cup[\mathbf{i}]$.

 \vskip 5pt  At last let us solve the same polynomial in Example 3.8 of
\cite{JO1}.

\vskip 5pt {\bf Example 3.} Find all zeros of $p(z) = z^6 +
\mathbf{j}z^5 + \mathbf{i}z^4 - z^2 - \mathbf{j}z - \mathbf{i}.$\\

{\it Solution 1.} We first use the method in this paper to do this.

We have $\mathbf{i}p(z)=\mathbf{i}z^6 + \mathbf{i}\mathbf{j}z^5 -z^4
- \mathbf{i}z^2 -  \mathbf{i}\mathbf{j}z +1.$ Then
$$\mathbf{i}p(t)=(\mathbf{i}t^6  -t^4 - \mathbf{i}t^2
+1)+(\mathbf{i}t^ 5 - \mathbf{i}t)\mathbf{j},$$
 $$g=t^4-1,\,\,\,g_1=\mathbf{i}t^2-1,\,\,\, g_2=\mathbf{i}t,$$ $$\tilde g=g_1\bar g_1+g_2\bar
g_2=(\mathbf{i}t^2-1)(-\mathbf{i}t^2-1)+\mathbf{i}t\cdot
(-\mathbf{i}t)=t^4+t^2+1.$$ The all distinct zeros (real or complex)
for $g$ and $\tilde g$ are respectively
 $\{1, -1, \mathbf{i},
-\mathbf{i}\}$ and $\{ e^{-\mathbf{i}\frac{\pi}{3}},
e^{\mathbf{i}\frac{\pi}{3}}, e^{-\mathbf{i}\frac{2\pi}{3}},
e^{\mathbf{i}\frac{2\pi}{3}}\}$. Now we take

 $\eta_1=e^{-\mathbf{i}\frac{\pi}{3}}=1/2-\sqrt{-3}/2$

$\eta_2=e^{-\mathbf{i}\frac{2\pi}{3}}=-1/2-\sqrt{-3}/2$,

$g_1(\bar
\eta_1)=\mathbf{i}e^{\mathbf{i}\frac{2\pi}{3}}-1=-1-\frac{\sqrt{3}}{2}-\frac{1}{2}\mathbf{i}$,
$|g_1(\bar \eta_1)|^2=2+\sqrt{3}$;

$g_2(\bar
\eta_1)=\mathbf{i}e^{\mathbf{i}\frac{\pi}{3}}=-\frac{\sqrt{3}}{2}+\frac{1}{2}\mathbf{i}$,
$|g_2(\bar \eta_1)|^2=1$;

$g_1(\bar
\eta_2)=\mathbf{i}e^{\mathbf{i}\frac{4\pi}{3}}-1=-1+\frac{\sqrt{3}}{2}-\frac{1}{2}\mathbf{i}$,
$|g_1(\bar \eta_2)|^2=2-\sqrt{3}$;

$g_2(\bar
\eta_2)=\mathbf{i}e^{\mathbf{i}\frac{2\pi}{3}}=-\frac{\sqrt{3}}{2}-\frac{1}{2}\mathbf{i}$,
$|g_2(\bar \eta_1)|^2=1$.\\
 So by formula
(\ref{formulaomega}), we get

$$\aligned \omega_1=&\frac{1}{3+\sqrt{3}}\{(\frac{1}{2}+\frac{\sqrt{3}}{2}\mathbf{i})+(2+\sqrt{3})(\frac{1}{2}
 -\frac{\sqrt{3}}{2}\mathbf{i})
 \\ &\hskip 1cm +2(-\frac{\sqrt{3}}{2}+\frac{1}{2}\mathbf{i})(-1-\frac{\sqrt{3}}{2}+\frac{1}{2}\mathbf{i})
 (-\frac{\sqrt{3}}{2})\mathbf{k}\}\\
=&\frac{1}{2}-\frac{1}{2}\mathbf{i}-\frac{1}{2}\mathbf{j}-\frac{1}{2}\mathbf{k}\endaligned$$

 $$\aligned \omega_2=&\frac{1}{3-\sqrt{3}}\{(-\frac{1}{2}+\frac{\sqrt{3}}{2}\mathbf{i})+(2-\sqrt{3})(-\frac{1}{2}
 -\frac{\sqrt{3}}{2}\mathbf{i})
 \\ &\hskip 1cm +2(-\frac{\sqrt{3}}{2}-\frac{1}{2}\mathbf{i})(-1+\frac{\sqrt{3}}{2}+\frac{1}{2}\mathbf{i})
 (-\frac{\sqrt{3}}{2})\mathbf{k}\}\\
 =&-\frac{1}{2}+\frac{1}{2}\mathbf{i}-\frac{1}{2}\mathbf{j}-\frac{1}{2}\mathbf{k}.\endaligned$$

Hence, the set of roots of $p(z)$ is
  $$\{1, -1\}\cup\{\frac{1}{2}-\frac{1}{2}\mathbf{i}-\frac{1}{2}\mathbf{j}-\frac{1}{2}\mathbf{k},
  -\frac{1}{2}+\frac{1}{2}\mathbf{i}-\frac{1}{2}\mathbf{j}-\frac{1}{2}\mathbf{k} \}\cup
[\mathbf{i}].$$\\

{\it Solution 2.} We   use the method in \cite{JO1} to redo this.\\

$\textit{Step 1}$ \ Compute the companion polynomial $q_{12}(x)$ by
formula (\ref{companion}):

 $$\aligned
              q_{12}(x)=&x^{12}+\mathbf{j}x^{11}+\mathbf{i}x^{10}-x^8-\mathbf{j}x^7-\mathbf{i}x^6 \\
              &-\mathbf{j}x^{11}+x^{10}-\mathbf{ji}x^9+\mathbf{j}x^7-x^6+\mathbf{ji}x^5 \\
              &-\mathbf{i}x^{10}-\mathbf{k}x^9+x^8+\mathbf{i}x^6+\mathbf{k}x^5-x^4 \\
              &-x^8-\mathbf{j}x^7-\mathbf{i}x^6+x^4+\mathbf{j}x^3+\mathbf{i}x^2\\
              &+\mathbf{j}x^7-x^6+\mathbf{ji}x^5-\mathbf{j}x^3+x^2-\mathbf{ji}x \\
              &+\mathbf{i}x^6+\mathbf{k}x^5-x^4-\mathbf{i}x^2-\mathbf{k}x+1\endaligned$$ $$\aligned
            =&x^{12}+0+x^{10}+0-x^8+0-2x^6+0-x^4+0+x^2+1\\
            =&x^{12}+x^{10}-x^8-2x^6-x^4+x^2+1.\endaligned$$

$\textit{Step 2 }$ \ Compute the zeros of   $q_{12}(x)$.\\
$$\aligned q_{12}(x)=&x^8(x^4+x^2-1)-x^6-x^4+x^2-x^6+1\\
=&x^8(x^4+x^2-1)-x^2(x^4+x^2-1)-x^6+1\\
=&(x^4+x^2-1)(x^8-x^2)-(x^6-1)\\
=&(x^4-1)(x^2+1)(x^6-1).\endaligned$$ The $12$ zeros of $q_{12}$ are
$$1, 1, -1, -1, \mathbf{i}, \mathbf{i}, -\mathbf{i}, -\mathbf{i},
\frac{1}{2}+\frac{\sqrt{3}}{2}\mathbf{i},
\frac{1}{2}-\frac{\sqrt{3}}{2}\mathbf{i},
-\frac{1}{2}+\frac{\sqrt{3}}{2}\mathbf{i},
-\frac{1}{2}-\frac{\sqrt{3}}{2}\mathbf{i}.$$

\

$\textit{Step 3 }$\  For $1$ and $-1$, we get isolated zeros: $1$,
$-1$.

For $x=\mathbf{i},-\mathbf{i}$, then
$$\begin{array}{ll}
  \mathbf{i}^0  =0\mathbf{i}+1 & (-\mathbf{i})^0=0(-i)+1 \\
  \mathbf{i}^1  =1\mathbf{i}+0 & (-\mathbf{i})^1=1(-\mathbf{i})+0\\
  \mathbf{i}^2  =0\mathbf{i}+(-1) & (-\mathbf{i})^2=0(-\mathbf{i})+(-1) \\
  \mathbf{i}^3  =(-1)\mathbf{i}+0 & (-\mathbf{i})^3=(-1)(-\mathbf{i})+0 \\
  \mathbf{i}^4  =0\mathbf{i}+1 & (-\mathbf{i})^4=0(-\mathbf{i})+1 \\
  \mathbf{i}^5  =1\mathbf{i}+0 & (-\mathbf{i})^5=1(-\mathbf{i})+0 \\
  \mathbf{i}^6  =0\mathbf{i}+(-1)  & (-\mathbf{i})^6=0(-\mathbf{i})+(-1)
  \\
  \\
    A(\mathbf{i})=  1\cdot 0+\mathbf{j}\cdot 1
    +\mathbf{i}\cdot 0 & A(-\mathbf{i}) = 1\cdot 0+\mathbf{j}\cdot 1
    +\mathbf{i}\cdot 0\\ \hskip .3cm +0\cdot (-1)
      +(-1)\cdot 0
    & \hskip .3cm +0\cdot (-1)
      +(-1)\cdot 0
    \\  \hskip .3cm +(-\mathbf{j})\cdot 1+(-\mathbf{i})\cdot 0=0 & \hskip .3cm +(-\mathbf{j})\cdot 1 +(-\mathbf{i})\cdot 0=0
   \\ \\
  v=\overline{A(\mathbf{i})}B(\mathbf{i})=0 &  v=\overline{A(-\mathbf{i})}B(-\mathbf{i})=0.
\end{array}$$
They produce the same  spherical zero $[\mathbf{i}]$ since
$[-\mathbf{i}]=[\mathbf{i}]$.

\

For $x=\frac{1}{2}+\frac{\sqrt{3}}{2}\mathbf{i}$,
$\frac{1}{2}-\frac{\sqrt{3}}{2}\mathbf{i}$, we have\\
$$\small
\begin{array}{ll}
 (\frac{1}{2}+\frac{\sqrt{3}}{2}\mathbf{i})^0=0(\frac{1}{2}+\frac{\sqrt{3}}{2}\mathbf{i})+1  & (\frac{1}{2}
 -\frac{\sqrt{3}}{2}\mathbf{i})^0=0(\frac{1}{2}-\frac{\sqrt{3}}{2}\mathbf{i})+1 \\
  (\frac{1}{2}+\frac{\sqrt{3}}{2}\mathbf{i})^1=1(\frac{1}{2}+\frac{\sqrt{3}}{2}\mathbf{i})+0 & (\frac{1}{2}
  -\frac{\sqrt{3}}{2}\mathbf{i})^1=1(\frac{1}{2}-\frac{\sqrt{3}}{2}\mathbf{i})+0\\
  (\frac{1}{2}+\frac{\sqrt{3}}{2}\mathbf{i})^2=1(\frac{1}{2}+\frac{\sqrt{3}}{2}\mathbf{i})+(-1) & (\frac{1}{2}
  -\frac{\sqrt{3}}{2}\mathbf{i})^2=1(\frac{1}{2}-\frac{\sqrt{3}}{2}\mathbf{i})+(-1) \\
  (\frac{1}{2}+\frac{\sqrt{3}}{2}\mathbf{i})^3=0(\frac{1}{2}+\frac{\sqrt{3}}{2}\mathbf{i})+(-1) & (\frac{1}{2}
  -\frac{\sqrt{3}}{2}\mathbf{i})^3=0(\frac{1}{2}-\frac{\sqrt{3}}{2}\mathbf{i})+(-1) \\
  (\frac{1}{2}+\frac{\sqrt{3}}{2}\mathbf{i})^4=(-1)(\frac{1}{2}+\frac{\sqrt{3}}{2}\mathbf{i})+0 & (\frac{1}{2}
  -\frac{\sqrt{3}}{2}\mathbf{i})^4=(-1)(\frac{1}{2}-\frac{\sqrt{3}}{2}\mathbf{i})+0 \\
  (\frac{1}{2}+\frac{\sqrt{3}}{2}\mathbf{i})^5=(-1)(\frac{1}{2}+\frac{\sqrt{3}}{2}\mathbf{i})+1 &
  (\frac{1}{2}-\frac{\sqrt{3}}{2}\mathbf{i})^5=(-1)(\frac{1}{2}-\frac{\sqrt{3}}{2}\mathbf{i})+1\\
  (\frac{1}{2}+\frac{\sqrt{3}}{2}\mathbf{i})^6=0(\frac{1}{2}+\frac{\sqrt{3}}{2}\mathbf{i})+1  &
  (\frac{1}{2}-\frac{\sqrt{3}}{2}\mathbf{i})^6=0(\frac{1}{2}-\frac{\sqrt{3}}{2}\mathbf{i})+1\\
  \\
    A(x)=1\cdot 0+\mathbf{j}\cdot -1 +\mathbf{i}\cdot -1&   A(x)= 1\cdot 0+\mathbf{j}\cdot -1   +\mathbf{i}\cdot -1 \\
    \hskip .3cm +0\cdot 0-1\cdot 1+-\mathbf{j}\cdot 1   +-\mathbf{i}\cdot 0&\hskip .3cm +0\cdot 0+\cdots\\
    \hskip .3cm  =-1-\mathbf{i}-2\mathbf{j} & \hskip .3cm  =-1-\mathbf{i}-2\mathbf{j} \end{array}$$
  $$\small
\begin{array}{ll}
    B(x)= 1\cdot 1+\mathbf{j}\cdot 1    +\mathbf{i}\cdot 0&   B(x)= 1\cdot 1+\mathbf{j}\cdot 1   \\
    \hskip .3cm -1\cdot -1+-\mathbf{j}\cdot 0   +-\mathbf{i}\cdot 1+0\cdot -1 & \hskip .3cm
    +\mathbf{i}\cdot 0 +\cdots \\\hskip .3cm +0\cdot -1=2-\mathbf{i}+\mathbf{j}  & \hskip .3cm =2-\mathbf{i}+\mathbf{j}
     \\ \\
  v=\overline{A(x)}B(x)=-3+3\mathbf{i}+3\mathbf{j}+3\mathbf{k} &
  v=\overline{A(x)}B(x)=-3+3\mathbf{i}+3\mathbf{j}+3\mathbf{k}\\
  |w|=3\sqrt{3}&|w|=3\sqrt{3}.
\end{array}$$
They  produce the same  isolated zero:
$\frac{1}{2}-\frac{\frac{\sqrt{3}}{2}}{3\sqrt{3}}\cdot
3\mathbf{i}-\frac{\frac{\sqrt{3}}{2}}{3\sqrt{3}}\cdot
3\mathbf{j}-\frac{\frac{\sqrt{3}}{2}}{3\sqrt{3}}\cdot 3\mathbf{k}$,
which is $0.5-0.5\mathbf{i}-0.5\mathbf{j}-0.5\mathbf{k}$.

\

Finally for $x=-\frac{1}{2}+\frac{\sqrt{3}}{2}\mathbf{i}$,
$-\frac{1}{2}-\frac{\sqrt{3}}{2}\mathbf{i}$, we still have to finish the following work:\\
$\small
\begin{array}{cc}
 (-\frac{1}{2}+\frac{\sqrt{3}}{2}\mathbf{i})^0=0(-\frac{1}{2}+\frac{\sqrt{3}}{2}\mathbf{i})+1  & (-\frac{1}{2}
 -\frac{\sqrt{3}}{2}\mathbf{i})^0=0(-\frac{1}{2}-\frac{\sqrt{3}}{2}\mathbf{i})+1 \\
  (-\frac{1}{2}+\frac{\sqrt{3}}{2}\mathbf{i})^1=1(-\frac{1}{2}+\frac{\sqrt{3}}{2}\mathbf{i})+0 & (-\frac{1}{2}
  -\frac{\sqrt{3}}{2}\mathbf{i})^1=1(-\frac{1}{2}-\frac{\sqrt{3}}{2}\mathbf{i})+0\\
  (-\frac{1}{2}+\frac{\sqrt{3}}{2}\mathbf{i})^2=-1(-\frac{1}{2}+\frac{\sqrt{3}}{2}\mathbf{i})+(-1) & (-\frac{1}{2}
  -\frac{\sqrt{3}}{2}\mathbf{i})^2=-1(-\frac{1}{2}-\frac{\sqrt{3}}{2}\mathbf{i})+(-1) \\
  (-\frac{1}{2}+\frac{\sqrt{3}}{2}\mathbf{i})^3=0(-\frac{1}{2}+\frac{\sqrt{3}}{2}\mathbf{i})+1 & (-\frac{1}{2}
  -\frac{\sqrt{3}}{2}\mathbf{i})^3=0(-\frac{1}{2}-\frac{\sqrt{3}}{2}\mathbf{i})+1\\
  (-\frac{1}{2}+\frac{\sqrt{3}}{2}\mathbf{i})^4=1(-\frac{1}{2}+\frac{\sqrt{3}}{2}\mathbf{i})+0 & (-\frac{1}{2}
  -\frac{\sqrt{3}}{2}\mathbf{i})^4=1(-\frac{1}{2}-\frac{\sqrt{3}}{2}\mathbf{i})+0 \\
  (-\frac{1}{2}+\frac{\sqrt{3}}{2}\mathbf{i})^5=(-1)(-\frac{1}{2}+\frac{\sqrt{3}}{2}\mathbf{i})+(-1) &
  (-\frac{1}{2}-\frac{\sqrt{3}}{2}\mathbf{i})^5=(-1)(-\frac{1}{2}-\frac{\sqrt{3}}{2}\mathbf{i})+-1 \\
  (-\frac{1}{2}+\frac{\sqrt{3}}{2}\mathbf{i})^6=0(-\frac{1}{2}+\frac{\sqrt{3}}{2}\mathbf{i})+1  & (-\frac{1}{2}
  -\frac{\sqrt{3}}{2}\mathbf{i})^6=0(-\frac{1}{2}-\frac{\sqrt{3}}{2}\mathbf{i})+1  \\
  \\
   {\small{\small \begin{array}{c}
    A(x)=1\cdot 0+\mathbf{j}\cdot -1
    +\mathbf{i}\cdot 1+0\cdot 0 \\ \hskip .3cm
    -1\cdot -1+-\mathbf{j}\cdot 1
    +-\mathbf{i}\cdot 0\\ \hskip .3cm =1+\mathbf{i}-2\mathbf{j}
  \end{array}}}
   & {\small   {\small                                      \begin{array}{c}
    A(x)=
 1\cdot 0+\mathbf{j}\cdot -1
    +\mathbf{i}\cdot 1\\ \hskip .3cm
    +0\cdot 0+\cdots \\ \hskip .3cm =1+\mathbf{i}-2\mathbf{j}
  \end{array}}}  \\ \\
\begin{array}{c}
    B(x)=
 1\cdot 1+\mathbf{j}\cdot -1
    +\mathbf{i}\cdot 0 +0\cdot 1\\ \hskip .3cm
    -1\cdot -1+-\mathbf{j}\cdot 0
    +-\mathbf{i}\cdot 1\\ \hskip .3cm  =2-\mathbf{i}-\mathbf{j}
  \end{array}& \begin{array}{c}
    B(x)=
 1\cdot 1+\mathbf{j}\cdot -1
    +\mathbf{i}\cdot 0\\ \hskip .3cm +0\cdot 1
    +\cdots \\ \hskip .3cm =2-\mathbf{i}-\mathbf{j}
  \end{array}
  \\ \\
  v=\overline{A(x)}B(x)=3-3\mathbf{i}+3\mathbf{j}+3\mathbf{k} &
  v=\overline{A(x)}B(x)=3-3\mathbf{i}+3\mathbf{j}+3\mathbf{k}\\ \\
  |w|=3\sqrt{3}&|w|=3\sqrt{3}.
\end{array}$\\
 By formula (\ref{formula4}), They produce the same isolated zero:
$-0.5+0.5\mathbf{i}-0.5\mathbf{j}-0.5\mathbf{k}$. The set of roots
of $p(z)$ is
  $$\{1,\,\, -1, \,\, 0.5-0.5\mathbf{i}-0.5\mathbf{j}-0.5\mathbf{k},\,\,  -0.5+0.5\mathbf{i}-0.5\mathbf{j}-0.5\mathbf{k}
  \}\cup
[\mathbf{i}].$$

In  Example $3.8$ of \cite{JO1}, the authors said that the $12$
zeros of $q_{12}$ were 1 (twice), -1 (twice), $\pm \mathbf{i}$
(twice each), $0.5(\pm 1\pm \mathbf{i})$. We point out that $0.5(\pm
1\pm \mathbf{i})$ should be corrected as $0.5(\pm 1\pm
\sqrt{3}\mathbf{i})$.

 \section{Numerical consideration and algorithm comparison}
 The polynomial in Example 1 has the property that its discriminant polynomial $\tilde p$
  can be factored by hands as $(t-\mathbf{i})(t+\mathbf{i})(t-e^{\mathbf{i}\frac{\pi}{4}})
  (t-e^{\mathbf{i}\frac{3\pi}{4}})(t-e^{\mathbf{i}\frac{5\pi}{4}})(t-e^{\mathbf{i}\frac{7\pi}{4}})$.
  In general, one can not always expect to  get the zeros of the discriminant
  polynomial by  hands, one has to rely on machine
  computations. In Algorithm 1, if we
compute the zeros of the discriminant polynomial, $\tilde p(x)$, of
Example 3 in Section 4 by MATLAB 7.12.0(R2011a), we find the 12
zeros are as following:
\begin{equation}\label{root table}
\begin{array}{ccc}
    1 & -1.000000000000001 & +0.000000002066542i \\
    2 & -1.000000000000001 & -0.000000002066542i \\
    3 & -0.500000000000000 & +0.866025403784440i \\
    4 & -0.500000000000000 & -0.866025403784440i  \\
    5 & \ \ 0.999999990102304 & +0.000000000000000i \\
    6 & \ \ 1.000000009897694 & -0.000000000000000i \\
    7 & \ \ 0.500000000000000 & +0.866025403784439i \\
    8 & \ \ 0.500000000000000 & -0.866025403784439i\\
    9 & \ \ 0.000000000016075 & +1.000000008531051i \\
    10 & \ \ 0.000000000016075 & -1.000000008531051i \\
    11 & -0.000000000016074 & +0.999999991468949i \\
    12 & -0.000000000016074 & -0.999999991468949i
  \end{array}
  \end{equation}
  which are nearly same as Table 1 in \cite{JO1}. This is simply because
our $\tilde p(t)$ is exactly the companion polynomial $q_{12}(t)$.
Hence,  the similar measures to \cite{JO1} can be made to obtain
machine precision for the zeros with multiplicity 2(e.g., by
Newton's method). It is interesting to note that the discriminant
polynomial given in this paper is always equal to the companion
polynomial after considering the variable as a real variable.

In Algorithm 1, if we have found by MATLAB the zeros of discriminant
polynomial, we usually by comparing
$|f_1(\eta_i)|^2+|f_2(\eta_i)|^2$ with
$|f_1(\overline{\eta_i})|^2+|f_2(\overline{\eta_i})|^2$ decide which
one $\omega_i$ should take between (\ref{zero1}) and (\ref{zero2}).
If $|f_1(\eta_i)|^2+|f_2(\eta_i)|^2$ is greater, then $\omega_i$
takes (\ref{zero1}), otherwise takes (\ref{zero2}). However, our
Algorithm 1' does not need to make such a decision.

In Step 2 of Algorithm 1, we need decide whether a zero $\xi$ is
real. In our experience (the same as \cite{JO1}), a test of the form
$|{\rm{Im}}\xi|<10^{-5}$ is appropriate. There is another delicate
decision to make in our Algorithm 1.  That is, in Step 2 one has to
decide whether a zero $\eta$ satisfies $f_1(\eta)=f_2(\eta)=\bar
f_1(\eta)=\bar f_2(\eta)=0$. For this, a test for $f(\eta)=0$ can be
carried out in the form $|f(\eta)|<10^{-10}$(This was also used to
test $v=0$ in \cite{JO1}). Our Algorithm 1' avoids making this
delicate decision.

Both Algorithm 1 and Algorithm 2 need to find all zeros of the
companion polynomial(note that the discriminant polynomial is always
equal to the companion polynomial). Unlike Algorithm 2,  our
algorithms  no longer need the  use of any iterations. Using our
algorithms, one can easily produce all zeros of the simple
quaternionic polynomial from the zeros of the companion polynomial.
Our algorithm has less workload, as shown in Section 4.
 This has a
great advantage for a simple quaternionic polynomial with high
degree. In fact, when one use Iteration (\ref{eq: iteration1}) to
compute $A(z)$ and $B(z)$ for a nonreal complex root $z$ with
$|z|\neq 1$, if the degree of the simple quaternionic polynomial
considered is high, the workload is huge, even interferes one's
deciding whether $v$ is zero.

In [Page 252, 9] D. Janovsk$\acute{a}$ and  G. Opfer said  ``We made
some hundred tests with polynomials $p_n$ of degree $n \leq 50$ with
random integer coefficients in the range $[-5, 5]$ and with real
coefficients in the range $[0, 1]$.  In all cases we found only
(nonreal) isolated zeros $z$. The test cases showed $|p_n(z)|
\thickapprox 10^{-13}$. Real zeros and spherical zeros did not show
up. If n is too large, say $n \thickapprox 100$, then usually it is
not any more possible to find all zeros of the companion polynomial
by standard means (say roots in MATLAB) because the coefficients of
the companion polynomial will be too large." But this can be easily
avoid by using Algorithm 1' in many cases.  For example, to solve a
real coefficient polynomial of degree of $50$, we can use Corollary
3 to do this easily.

\vskip 5pt {\bf Example 4.} Find all zeros of $p(z) = z^{1000} -2$
in $\mathbb{H}$.

{\it Solution.} We know that the complex solution set of $p(t)$ is
$$\{\pm\sqrt[1000]2\}\cup_{k=1}^{499}\{\eta_k=\sqrt[1000]2e^{\frac{k\pi\mathbf{i}}{500}}, \bar\eta_k\}.$$
Applying Corollary 3, we obtain the solution set for $p(z)$:
$$\{\pm1\}\cup_{k=1}^{499}[\sqrt[1000]2e^{\frac{k\pi\mathbf{i}}{500}}].$$

According to the above  comments D. Janovsk$\acute{a}$ and  G. Opfer
made on their method, it is impossible to find all solutions of this
example since the degree of $p(z)$ is much larger than $100$ and all
roots except for $\pm \sqrt[1000]2$ are spherical.

 \

\section{Conclusion}

The method introduced in Section 2  can also be used to find all
zeros of simple polynomials with quaternionic coefficients located
on only right side of the power of the variable, since
$x^nq_n+\cdots+xq_1+q_0=0$ is equivalent to $\overline{q_n}\
\overline{x}^n+\cdots+\overline{q_1}\
\overline{x}+\overline{q_0}=0$. In retrospect, Eilenberg and Niven
in \cite{EN} for the first time proved that quaternionic
$q_nx^n+\cdots+q_1x+q_0=0$ has always a root, but the method they
used is topological. For this result, here we have provided
 a new  alternate proof, without using
any topological tools. Based on the derived polynomial and
discriminant polynomial introduced in this paper, we have given the
zeros in an explicit form for a simple quaternionic polynomial.
Comparing with the methods provided in \cite{JO1} to seek roots of
$q_nx^n+\cdots+q_1x+q_0=0$, our method is different, and has great
advantages. Using our method, we have recovered several known
results, and deduced several very interesting consequences
concerning seeking the zeros of a simple quaternionic polynomial
which do not seem to be deduced easily from the results in
\cite{JO1}.

We would like to conclude our paper by making the following remark.
We hope the method in the proof of our main theorem can be useful to
seeking a method to find all solutions of a quaternionic polynomial
equation with coefficients on both sides of the powers of the
variable.

\newcommand{\G}{\Gamma}
\newcommand{\ga}{\gamma}

\vspace{1mm}

{\small \noindent L.G. Feng:   Department of Mathematics and System
Science, National University of Defense Technology, Changsha 410073,
Hunan, P.R.China. Email: fenglg2002@sina.com}

{\small\vspace{0.2cm} \noindent K.M. Zhao: Department of
Mathematics, Wilfrid Laurier University, Waterloo, ON, Canada N2L
3C5, and College of Mathematics and Information Science, Hebei
Normal (Teachers) University, Shijiazhuang 050016, Hebei, P. R.
China. Email: kzhao@wlu.ca}


\vspace{0.5cm}


\begin{thebibliography}{99999}
\bibitem{A}  S.L. Adler, {\em Quaternionic quantum field theory}, Commun. Math.
Phys., 104(1986), pp. 611-656.
\bibitem{BM} N. Le Bihan and J. Mars,  {\em Singular value decomposition of quaternion
 matrices:
  a new tool for vector-sensor signal processing},
Signal Process, 84(2004), pp. 1177--1199.
\bibitem{EN} S. Eilenberg
and I. Niven, {\em The 'Fundamental Theorem of Algebra'  for
 quaternions}, Bull. Amer. Math. Soc., 50(1944), pp. 246-248.


\bibitem{GH} A. Gsponer and J.-P. Hurni, {\em Quaternions in Mathematical Physics (2):
Analytical Bibliography}, Independent Scientific Research Institute
report ISRI-05-05.26, 2008; also available online at
http://arxiv.org/abs/math-ph/0511092v3.
\bibitem{GSt}  G. Gentili and C. Stoppato, {\em Zeros
of regular functions and polynomials of a quaternionic variable},
Michigan Math. J., 56(2008), pp. 655-667.
\bibitem{GS} G. Gentili and D. C. Struppa, {\em On the
multiplicity of zeros of polynomials with quaternionic
coefficients}, Milan J. Math., 76(2007), pp. 1-10.
\bibitem{GSV}  G. Gentili, D.C. Struppa and F. Vlacci, {\em The
fundamental theorem of algebra for Hamilton and Cayley numbers},
Math. Z., 259(2008), pp. 895-902.

\bibitem{JO5}D. Janovsk$\acute{a}$ and G. Opfer,
{\em Linear equations in quaternions}, Numerical mathematics and
advanced applications,  Springer, Berlin, 2006, pp. 945-953.
\bibitem{JO1} D. Janovsk$\acute{a}$ and  G. Opfer, {\em  A note on
 the computation of all zeros of simple
 quaternionic polynomials}, SIAM J. Numer. Anal.,
48(1)(2010), pp. 244-256.
\bibitem{JO2} D. Janovsk$\acute{a}$ and G. Opfer, {\em The classification
and  the computation of the zeros of quaternionic two-sided
polynomials}, Numer. Math., 115(1)(2010), pp. 81-100.
\bibitem{JO3}  D. Janovsk$\acute{a}$ and  G. Opfer, {\em Computing
quaternionic  roots in Newton's method}, Electron. Trans. Numer.
Anal., 26(2007), pp. 82-102.
\bibitem{JO4}  D. Janovsk$\acute{a}$ and  G.
Opfer, {\em Linear equations in quaternionic variables}, Mitt. Math.
Ges. Hamburg, 27(2008), pp. 223-234.

\bibitem{DDL}  S. De Leo,  G. Ducati and V. Leonardi,  {\em Zeros of
unilateral  quaternionic polynomials}, Electron. J. Linear Algebra,
15(2006), pp. 297-313.


\bibitem{N} I. Niven,  {\em Equations in quaternions}, Amer. Math. Monthly, 48(1941),
pp. 654-661.
\bibitem{N2}I. Niven, {\em The roots of a quaternion}, Amer. Math. Monthly, 49(1942),
pp. 386-388.

\bibitem{PS}  A. Pogorui and M. Shapiro, {\em On the structure of the set of
zeros  of quaternionic polynomials}, Complex Var. Elliptic Funct.,
49(2004), pp. 379-389.

\bibitem{PW}  S. Pumpl$\ddot{u}$n and S. Walcher,  {\em On the zeros of polynomials
over quaternions}, Comm. Algebra, 30(2002), pp. 4007-4018.


\bibitem{SPV}  R. Ser$\hat{o}$dio, E. Pereira and   J. Vit$\acute{o}$ria, {\em Computing the zeros
of quaternionic polynomials}, Comput. Math. Appl., 42(2001), pp.
1229-1237.

\bibitem{SS} R. Ser$\hat{o}$dio and L.S. Siu, {\em Zeroes of
quaternion polynomials}, Appl. Math. Lett., 14(2)(2001), pp.
237-239.


\bibitem{T} N. Topuridze,{\em  On the roots of polynomials over division algebras}, Georgian Math. J., 10(4)(2003),
pp. 745-762.


\bibitem{ZFZ} F.Z. Zhang, {\em Quaternions and
matrices of quaternions}, Linear Algebra Appl., 251(1997), pp.
21-57.
\bibitem{ZY} F.Z. Zhang and Y. Wei, {\em Jordan canonical form of a partitioned
complex  matrix and its application to real quaternion matrices},
Comm.  Algebra, 29(6)(2001), pp. 2363-2375.
\end{thebibliography}
\end{document}